\documentclass[12pt,a4paper]{amsart}

\usepackage{amssymb}

\theoremstyle{plain}

\newtheorem{theorem}{Theorem}[section]
\newtheorem{proposition}[theorem]{Proposition}
\newtheorem{lemma}[theorem]{Lemma}
\newtheorem{corollary}[theorem]{Corollary}

\theoremstyle{remark}
\newtheorem{remark}[theorem]{Remark}

\theoremstyle{definition}
\newtheorem*{problem*}{Problem}
\newtheorem*{acknowledgements*}{Acknowledgements}

\title[Abstract quotients of profinite groups]{Abstract quotients of
  profinite groups, after Nikolov and Segal}

\author{Benjamin Klopsch} \address{Mathematisches Institut der
  Heinrich-Heine-Universit\"at, Universit\"atsstr.\ 1, 40225
  D\"usseldorf, Germany}\email{klopsch@math.uni-duesseldorf.de}


\begin{document}

\maketitle

\begin{abstract}
  In this expanded account of a talk given at the Oberwolfach
  Arbeitsgemeinschaft ``Totally Disconnected Groups'', October 2014,
  we discuss results of Nikolay Nikolov and Dan Segal on abstract
  quotients of compact Hausdorff topological groups, paying special
  attention to the class of finitely generated profinite groups.  Our
  primary source is~\cite{NN_DS12}.  Sidestepping all difficult and
  technical proofs, we present a selection of accessible arguments to
  illuminate key ideas in the subject.
\end{abstract} 

\section{Introduction}
\noindent
\textbf{\S {\thesection}.1.} Many concepts and techniques in the
theory of finite groups depend intrinsically on the assumption that
the groups considered are a priori finite.  The theoretical framework
based on such methods has led to marvellous achievements, including --
as a particular highlight -- the classification of all finite simple
groups.  Notwithstanding, the same methods are only of limited use in
the study of infinite groups: it remains mysterious how one could
possibly pin down the structure of a general infinite group in a
systematic way.  Significantly more can be said if such a group comes
equipped with additional information, such as a structure-preserving
action on a notable geometric object.  A coherent approach to studying
restricted classes of infinite groups is found by imposing suitable
`finiteness conditions', i.e., conditions that generalise the notion
of being finite but are significantly more flexible, such as the group
being finitely generated or compact with respect to a natural
topology.

One rather fruitful theme, fusing methods from finite and infinite
group theory, consists in studying the interplay between an infinite
group $\Gamma$ and the collection of all its finite quotients.  In
passing, we note that the latter, subject to surjections, naturally
form a lattice that is anti-isomorphic to the lattice of finite-index
normal subgroups of~$\Gamma$, subject to inclusions.  In this context
it is reasonable to focus on $\Gamma$ being \emph{residually finite},
i.e., the intersection $\bigcap_{N \trianglelefteq_\mathrm{f} \Gamma}
N$ of all finite-index normal subgroups of~$\Gamma$ being trivial.
Every residually finite group $\Gamma$ embeds as a dense subgroup into
its profinite completion $\widehat{\Gamma} = \varprojlim_{N
  \trianglelefteq_\mathrm{f} \Gamma} \Gamma/N$.  The latter can be
constructed as a closed subgroup of the direct product $\prod_{N
  \trianglelefteq_\mathrm{f} \Gamma} \Gamma/N$ of finite discrete
groups and thus inherits the structure of a \emph{profinite group},
i.e., a topological group that is compact, Hausdorff and totally
disconnected.  The finite quotients of $\Gamma$ are in one-to-one
correspondence with the continuous finite quotients of
$\widehat{\Gamma}$.

Profinite groups also arise naturally in a range of other contexts,
e.g., as Galois groups of infinite field extensions or as open compact
subgroups of Lie groups over non-archimedean local fields, and it is
beneficial to develop the theory of profinite groups in some
generality.  At an advanced stage, one is naturally led to examine
more closely the relationship between the underlying abstract group
structure of a profinite group~$G$ and its topology.  In the following
we employ the adjective `abstract' to emphasise that a subgroup,
respectively a quotient, of $G$ is not required to be closed,
respectively continuous.  For instance, as $G$ is compact, an abstract
subgroup $H$ is open in~$G$ if and only if it is closed and of finite
index in~$G$.  A profinite group may or may not have (normal) abstract
subgroups of finite index that fail to be closed.  Put in another way,
$G$ may or may not have non-continuous finite quotients.  What can be
said about abstract quotients of a profinite group $G$ in general?
When do they exist and how `unexpected' can their features possibly
be?

To the newcomer, even rather basic groups offer some initial
surprises.  For instance, consider for a prime $p$ the procyclic group
$\mathbb{Z}_p = \varprojlim_{k \in \mathbb{N}} \mathbb{Z} / p^k
\mathbb{Z}$ of $p$-adic integers under addition.  Of course, its
proper continuous quotients are just the finite cyclic groups
$\mathbb{Z}_p / p^k \mathbb{Z}_p$.  As we will see below, the abelian
group $\mathbb{Z}_p$ fails to map abstractly onto $\mathbb{Z}$, but
does have abstract quotients isomorphic to $\mathbb{Q}$.

\bigskip

\noindent
\textbf{\S {\thesection}.2.}  In~\cite{NN_DS12}, Nikolov and Segal
streamline and generalise their earlier results that led, in
particular, to the solution of the following problem raised by
Jean-Pierre Serre in the 1970s: are finite-index abstract subgroups of
an arbitrary finitely generated profinite group $G$ always open
in~$G$?  Recall that a profinite group is said to be (topologically)
\emph{finitely generated} if it contains a dense finitely generated
abstract subgroup.  The problem, which can be found in later editions
of Serre's book on Galois cohomology~\cite[I.\S4.2]{JPS97}, was solved
by Nikolov and Segal in~2003.

\begin{theorem}[Nikolov, Segal~\cite{NN_DS07a}] \label{thm:Serre} Let
  $G$ be a finitely generated profinite group.  Then every
  finite-index abstract subgroup $H$ of $G$ is open in~$G$.
\end{theorem}

Serre had proved this assertion in the special case, where $G$ is a
finitely generated pro-$p$ group for some prime~$p$, by a neat and
essentially self-contained argument; compare Section~\ref{sec:Serre}.
The proof of the general theorem is considerably more involved and
makes substantial use of the classification of finite simple groups;
the same is true for several of the main results stated below.

The key theorem in~\cite{NN_DS12} concerns normal subgroups in finite
groups and establishes results about products of commutators of the
following type.  There exists a function $f \colon \mathbb{N}
\rightarrow \mathbb{N}$ such that, if $H \trianglelefteq \Gamma =
\langle y_1, \dots, y_r \rangle$ for a finite group $\Gamma$, then
every element of the subgroup $[H,\Gamma] = \langle [h,g] \mid h \in
H, g \in \Gamma \rangle$ is a product of $f(r)$ factors of the form
$[h_1,y_1] [h_2,y_1^{-1}] \cdots [h_{2r-1},y_r] [h_{2r},y_r^{-1}]$
with $h_1,\dots,h_{2r} \in H$.  Under certain additional conditions on
$H$, a similar conclusion holds based on the significantly weaker
hypothesis that $\Gamma = H \langle y_1,\dots,y_r \rangle = \Gamma'
\langle y_1,\dots,y_r \rangle$, where $\Gamma' = [\Gamma,\Gamma]$
denotes the commutator subgroup.  A more precise version of the result
is stated as Theorem~\ref{thm:main-finite} in
Section~\ref{sec:NS-finite-profinite}.  By standard compactness
arguments, one obtains corresponding assertions for normal subgroups
of finitely generated profinite groups; this is also explained in
Section~\ref{sec:NS-finite-profinite}.

Every compact Hausdorff topological group $G$ is an extension of a
compact connected group $G^\circ$, its identity component, by a
profinite group~$G/G^\circ$.  By the Levi--Mal'cev Theorem, e.g., see
\cite[Theorem~9.24]{KH_SM13}, the connected component $G^\circ$ is
essentially a product of compact Lie groups and thus relatively
tractable.  We state two results whose proofs require the new
machinery developed in~\cite{NN_DS12} that goes beyond the methods
used in~\cite{NN_DS07a,NN_DS07b,NN_DS11}.

\begin{theorem}[Nikolov,
  Segal~\cite{NN_DS12}] \label{thm:further-results-1} Let $G$ be a
  compact Hausdorff topological group. Then every finitely generated
  abstract quotient of $G$ is finite.
\end{theorem}

\begin{theorem}[Nikolov,
  Segal~\cite{NN_DS12}] \label{thm:further-results-2} Let $G$ be a
  compact Hausdorff topological group such that $G/G^\circ$ is
  topologically finitely generated.  Then $G$ has a countably infinite
  abstract quotient if and only if $G$ has an infinite
  virtually-abelian continuous quotient.
\end{theorem}

In a recent paper, Nikolov and Segal generalised their results in
\cite{NN_DS12} to obtain the following structural theorem.

\begin{theorem}[Nikolov,
  Segal~\cite{NN_DS14}] \label{thm:further-results-3} Let $G$ be a
  compact Hausdorff group such that $G/G^\circ$ is topologically
  finitely generated.  Then for every closed normal subgroup $H
  \trianglelefteq_\mathrm{c} G$ the abstract subgroup $[H,G] = \langle
  [h,g] \mid h \in H, g \in G \rangle$ is closed in~$G$.
\end{theorem}

There are many interesting open questions regarding the algebraic
properties of profinite groups; the introductory survey~\cite{NN12}
provides more information as well as a range of ideas and suggestions
for further investigation.  In Section~\ref{sec:coho} we look at one
possible direction for applying and generalising the results of
Nikolov and Segal, namely in comparing the abstract and the continuous
cohomology of a finitely generated profinite group.

\section{Serre's problem on finite abstract quotients of profinite
  groups} \label{sec:Serre}

\noindent \textbf{\S {\thesection}.1.} A profinite group $G$ is said
to be \emph{strongly complete} if every finite-index abstract subgroup
is open in~$G$.  Equivalently, $G$ is strongly complete if every
finite quotient of $G$ is a continuous quotient.  Clearly, such a
group is uniquely determined by its underlying abstract group
structure as it is its own profinite completion.

It is easy to see that there are non-finitely generated profinite
groups that fail to be strongly complete.  For instance, $G =
\varprojlim_{n \in \mathbb{N}} C_p^{\, n} \cong C_p^{\, \aleph_0}$,
the direct product of a countably infinite number of copies of a
cyclic group $C_p$ of prime order~$p$, has $2^{2^{\aleph_0}}$
subgroups of index $p$, but only countably many of these are open.
This can be seen by regarding the abstract group $G$ as a vector space
of dimension $2^{\aleph_0}$ over a field with $p$ elements.
Implicitly, this simple example uses the axiom of choice.

Without taking a stand on the generalised continuum hypothesis, it is
slightly more tricky to produce an example of two profinite groups
$G_1$ and $G_2$ that are non-isomorphic as topological groups, but
nevertheless abstractly isomorphic.  As an illustration, we discuss a
concrete instance.  In~\cite{JK_13}, Jonathan A.\ Kiehlmann classifies
more generally countably-based abelian profinite groups, up to
continuous isomorphism and up to abstract isomorphism.

\begin{proposition}
  Let $p$ be a prime.  The pro-$p$ group $G = \prod_{i=1}^\infty
  C_{p^i}$ is abstractly isomorphic to $G \times \mathbb{Z}_p$, but
  there is no continuous isomorphism between the groups $G$ and $G
  \times \mathbb{Z}_p$.
\end{proposition}

\begin{proof}
  The torsion elements form a dense subset in the group~$G$, but they
  do not in $G \times \mathbb{Z}_p$.  Hence the two groups cannot be
  isomorphic as topological groups.

  It remains to show that $G$ is abstractly isomorphic to $G \times
  \mathbb{Z}_p$.  Let $\mathfrak{U} \subseteq \mathcal{P}(\mathbb{N})
  = \{ T \mid T \subseteq \mathbb{N} \}$ be a non-principal
  ultrafilter.  This means: $\mathfrak{U}$ is closed under finite
  intersections and under taking supersets; moreover, whenever a
  disjoint union $T_1 \sqcup T_2$ of two sets belongs to
  $\mathfrak{U}$, precisely one of $T_1$ and $T_2$ belongs to
  $\mathfrak{U}$; finally, $\mathfrak{U}$ does not contain any
  one-element sets.  Then $\mathfrak{U}$ contains all co-finite
  subsets of~$\mathbb{N}$, and, in fact, one justifies the existence
  of $\mathfrak{U}$ by enlarging the filter consisting of all
  co-finite subsets to an ultrafilter; this process relies on the
  axiom of choice.

  Informally speaking, we use the ultrafilter $\mathfrak{U}$ as a tool
  to form a consistent limit.  Indeed, with the aid of $\mathfrak{U}$,
  we define
  \begin{equation*}
    \begin{split}
      \psi \colon & G \cong \prod_{i=1}^\infty
      \mathbb{Z}/p^i\mathbb{Z} \longrightarrow \mathbb{Z}_p \cong
      \varprojlim_j
      \mathbb{Z}/p^j\mathbb{Z}, \\
      & x=(x_1+p\mathbb{Z},x_2+p^2\mathbb{Z},\dots) \mapsto
      \varprojlim_j (x \psi_j),
    \end{split}
  \end{equation*}
  where
  \[
  x \psi_j = a + p^j \mathbb{Z} \in \mathbb{Z} / p^j \mathbb{Z} \qquad
  \text{if $U_j(a) = \{ k \geq j \mid x_k \equiv_{p^j} a \} \in
    \mathfrak{U}$}.
  \]
  Observe that
  $\mathfrak{U} \ni \{ k \mid k \geq j \} = U_j(0) \sqcup \dots \sqcup
  U_j(p^j-1)$
  to see that the definition of $\psi_j$ is valid.  It is a routine
  matter to check that $\psi$ is a (non-continuous) homomorphism from
  $G$ onto $\mathbb{Z}_p$.  Furthermore, $G$ decomposes as an abstract
  group into a direct product $G = \ker \psi \times H$, where
  $H = \overline{\langle (1,1,\dots) \rangle} \cong \mathbb{Z}_p$ is
  the closed subgroup generated by $(1,1,\ldots)$.  Consequently,
  $\ker \psi \cong G / H$ as an abstract group.  Finally, we observe
  that
  \[
  \vartheta \colon G \cong \prod_{i=1}^\infty \mathbb{Z}/p^i\mathbb{Z}
  \longrightarrow G \cong \prod_{i=1}^\infty \mathbb{Z}/p^i\mathbb{Z},
  \quad (x_i + p^i\mathbb{Z})_i \mapsto (x_i-x_{i+1} +
  p^i\mathbb{Z})_i
  \]
  is a continuous, surjective homomorphism with $\ker \vartheta = H$.
  This shows that 
  \[
  G \cong G/H \times H \cong G \times \mathbb{Z}_p
  \]
  as abstract groups.
\end{proof}

Theorem~\ref{thm:Serre} states that finitely generated profinite
groups are strongly complete: there are no unexpected finite quotients
of such groups.  What about infinite abstract quotients of finitely
generated profinite groups?  In view of Theorem~\ref{thm:Serre}, the
following proposition applies to all finitely generated profinite
groups.

\begin{proposition}
  Let $G$ be a strongly complete profinite group and $N
  \trianglelefteq G$ a normal abstract subgroup.  If $G/N$ is
  residually finite, then $N \trianglelefteq_\mathrm{c} G$ and $G/N$
  is a profinite group with respect to the quotient topology.
\end{proposition}

\begin{proof}
  The group $N$ is the intersection of finite-index abstract subgroups
  of~$G$.  Since $G$ is strongly complete, each of these is open and
  hence closed in~$G$.
\end{proof}

\begin{corollary} \label{cor:no-inf-res-fin}
  A strongly complete profinite group does not admit any countably
  infinite residually finite abstract quotients.
\end{corollary}

We remark that a strongly complete profinite group can have countably
infinite abstract quotients.  Indeed, we conclude this section with a
simple argument, showing that the procyclic group $\mathbb{Z}_p$ has
abstract quotients isomorphic to~$\mathbb{Q}$.  Again, this makes
implicitly use of the axiom of choice.

\begin{proposition}
  The procyclic group $\mathbb{Z}_p$ maps non-continuously onto
  $\mathbb{Q}$.
\end{proposition}

\begin{proof}
  Clearly, $\mathbb{Z}_p$ is a spanning set for the
  $\mathbb{Q}$-vector space $\mathbb{Q}_p$.  Thus there exists an
  ordered $\mathbb{Q}$-basis $x_0 = 1$, $x_1$, \ldots, $x_\omega$,
  \ldots\ for $\mathbb{Q}_p$ consisting of $p$-adic integers; here
  $\omega$ denotes the first infinite ordinal number.  Set
  \[
  y_0 = x_0 = 1, \quad y_i = x_i - p^{-i} \text{ for $1 \leq i <
    \omega$} \quad \text{and} \quad y_\lambda = x_\lambda \text{ for
    $\lambda \geq \omega$.}
  \]
  The $\mathbb{Q}$-vector space $\mathbb{Q}_p$ decomposes into a
  direct sum $\mathbb{Q} \oplus W = \mathbb{Q}_p$, where $W$ has
  $\mathbb{Q}$-basis $y_1, \dots, y_\omega, \dots$, inducing a natural
  surjection $\eta \colon \mathbb{Q}_p \rightarrow \mathbb{Q}$ with
  kernel~$W$.  By construction, $\mathbb{Z}_p + W = \mathbb{Q}_p$ so
  that $\mathbb{Z}_p / (\mathbb{Z}_p \cap W) \cong (\mathbb{Z}_p +
  W)/W \cong \mathbb{Q}$.
\end{proof}

To some extent this basic example is rather typical; cf.\
Theorems~\ref{thm:further-results-1} and \ref{thm:further-results-2}.
 
\bigskip

\noindent \textbf{\S {\thesection}.2.} We now turn our attention to
finitely generated profinite groups.  Around 1970 Serre discovered
that finitely generated pro-$p$ groups are strongly complete, and he
asked whether this is true for arbitrary finitely generated profinite
groups.  Here and in the following, $p$ denotes a prime.

\begin{theorem}[Serre] \label{thm:Serre-pro-p}
  Finitely generated pro-$p$ groups are strongly complete.
\end{theorem}

It is instructive to recall a proof of this special case of
Theorem~\ref{thm:Serre}, which we base on two auxiliary lemmata.

\begin{lemma} \label{lem:index} Let $G$ be a finitely generated
  pro-$p$ group and $H \leq_\mathrm{f} G$ a finite-index subgroup.
  Then $\lvert G : H \rvert$ is a power of~$p$.
\end{lemma}

\begin{proof}
  Indeed, replacing $H$ by its core in $G$, viz., the subgroup
  $\bigcap_{g \in G} H^g \trianglelefteq_\mathrm{f} G$, we may assume
  that $H$ is normal in~$G$.  Thus $G/H$ is a finite group of
  order~$m$, say, and the set $X = G^{\{m\}} = \{ g^m \mid g \in G \}$
  of all $m$th powers in~$G$ is contained in~$H$.  Being the image of
  the compact space $G$ under the continuous map $x \mapsto x^m$, the
  set $X$ is closed in the Hausdorff space~$G$.

  Since $G$ is a pro-$p$ group, it has a base of neighbourhoods of $1$
  consisting of open normal subgroups $N \trianglelefteq_\mathrm{o} G$
  of $p$-power index.  Hence $G/N$ is a finite $p$-group for every $N
  \trianglelefteq_\mathrm{o} G$.  Let $g \in G$.  Writing $m = p^r
  \widetilde{m}$ with $p \nmid \widetilde{m}$, we conclude that
  $g^{p^r} \in XN$ for every $N \trianglelefteq_\mathrm{o} G$.  Since
  $X$ is closed in $G$, this yields
  \[
  g^{p^r} \in \bigcap\nolimits_{N \trianglelefteq_\mathrm{o} G} XN = X
  \subseteq H.
  \]
  Consequently, every element of $G/H$ has $p$-power order, and
  $\lvert G : H \rvert = m = p^r$.
\end{proof}

For any subset $X$ of a topological group $G$ we denote by
$\overline{X}$ the topological closure of $X$ in~$G$.

\begin{lemma} \label{lem:comm} Let $G$ be a finitely generated pro-$p$
  group.  Then the abstract commutator subgroup $[G,G]$ is closed
  in~$G$, i.e., $[G,G] = \overline{[G,G]}$.
\end{lemma}

\begin{proof}
  We use the following fact about finitely generated nilpotent groups
  that can easily be proved by induction on the nilpotency class:
  \begin{quote}
    $(\dagger)$ if $\Gamma = \langle \gamma_1, \dots, \gamma_d \rangle$ is a
    nilpotent group then
    \[ 
    [\Gamma,\Gamma] = \{ [x_1,\gamma_1] \cdots [x_d,\gamma_d] \mid
    x_1, \dots, x_d \in \Gamma \}.
    \]
  \end{quote}

  Suppose that $G$ is topologically generated by $a_1, \dots, a_d$,
  i.e., $G = \overline{\langle a_1, \dots a_d \rangle}$.  Being the
  image of the compact space $G \times \dots \times G$ under the
  continuous map $(x_1,\dots,x_d) \mapsto [x_1,a_1] \dots [x_d,a_d]$,
  the set
  \[
  X = \{ [x_1,a_1] \dots [x_d,a_d] \mid x_1, \dots, x_d \in G \}
  \subseteq [G,G]
  \]
  is closed in the Hausdorff space~$G$.  Using $(\dagger)$, this yields
  \[
  \overline{[G,G]} = \bigcap\nolimits_{N \trianglelefteq_\mathrm{o} G}
  [G,G] N = \bigcap\nolimits_{N \trianglelefteq_\mathrm{o} G} X N = X
  \subseteq [G,G],
  \]
  and thus $[G,G] = \overline{[G,G]}$.
\end{proof}

The Frattini subgroup $\Phi(G)$ of a profinite group $G$ is the
intersection of all its maximal open subgroups.  If $G$ is a finitely
generated pro-$p$ group, then $\Phi(G) = \overline{G^p [G,G]}$ is open
in~$G$.  Here $G^p = \langle g^p \mid g \in G \rangle$ denotes the
abstract subgroup generated by the set $G^{\{p\}} = \{ g^p \mid g \in
G \}$ of all $p$th powers.

\begin{corollary} \label{cor:Frattini} Let $G$ be a finitely generated
  pro-$p$ group.  Then $\Phi(G)$ is equal to
  \[
  G^p [G,G] = G^{\{p\}} [G,G] = \{ x^p y \mid x \in G \text{ and } y
  \in [G,G] \}.
  \]
\end{corollary}

\begin{proof} 
  Indeed, $G^p [G,G] = G^{\{p\}} [G,G]$, because $G/[G,G]$ is abelian.
  By Lemma~\ref{lem:comm}, $[G,G]$ is closed in~$G$.  Being the image
  of the compact space $G \times [G,G]$ under the continuous map $(x,y)
  \mapsto x^p y$, the set $G^{\{p\}} [G,G]$ is closed in the Hausdorff
  space~$G$.  Thus $\Phi(G) = \overline{G^p [G,G]} = G^{\{p\}} [G,G]$.
\end{proof}

\begin{proof}[Proof of Theorem~\ref{thm:Serre-pro-p}]
  Let $H \leq_\mathrm{f} G$.  Replacing $H$ by its core in $G$, we may
  assume that $H \trianglelefteq_\mathrm{f} G$.  Using
  Lemma~\ref{lem:index}, we argue by induction on $\lvert G : H \rvert
  = p^r$.  If $r=0$ then $H = G$ is open in~$G$.  Now suppose that $r
  \geq 1$.
  
  From Corollary~\ref{cor:Frattini} we deduce that $M = H \Phi(G)$ is
  a proper open subgroup of~$G$.  Since $M$ is a finitely generated
  pro-$p$ group and $H \trianglelefteq_\mathrm{f} M$ with $\lvert M :
  H \rvert < \lvert G : H \rvert$, induction yields $H \leq_\mathrm{o}
  M \leq_\mathrm{o} G$.  Thus $H$ is open in~$G$.
\end{proof}

For completeness, we record another consequence of the proof of
Lemma~\ref{lem:comm} that can be regarded as a special case of
Corollary~\ref{cor:N=G} below.

\begin{corollary} \label{cor:N=G-pro-p}
  Let $G$ be a finitely generated pro-$p$ group and $N \trianglelefteq
  G$ a normal abstract subgroup.  If $N [G,G] = G$ then $N = G$.  
\end{corollary}

\begin{proof}
  Suppose that $N [G,G] = G$.  Then $N \Phi(G) = G$ and $N$ contains a
  set of topological generators $a_1, \ldots, a_d$ of $G$.  Arguing as
  in the proof of Lemma~\ref{lem:comm} we obtain $[G,G] \subseteq N$
  and thus $N = N [G,G] = G$.
\end{proof}

\bigskip

\noindent \textbf{\S {\thesection}.3.} There were several
generalisations of Serre's result to other classes of finitely
generated profinite groups, most notably by Brian Hartley~\cite{BH79}
dealing with poly-pronilpotent groups, by Consuelo Mart\'inez and Efim
Zelmanov~\cite{CM_EZ96} and independently by Jan Saxl and John S.\
Wilson~\cite{JS_JW97}, each team dealing with direct products of
finite simple groups, and finally by Segal~\cite{DS_00} dealing with
prosoluble groups.

The result of Mart\'inez and Zelmanov, respectively Saxl and Wilson,
on powers in non-abelian finite simple groups relies on the
classification of finite simple groups and leads to a slightly more
general theorem.  A profinite group $G$ is called \emph{semisimple} if
it is the direct product of non-abelian finite simple groups.

\begin{theorem}[Mart\'inez and Zelmanov~\cite{CM_EZ96}; Saxl and
  Wilson~\cite{JS_JW97}] \label{thm:SaWiMaZe} Let $G = \prod_{i \in
    \mathbb{N}} S_i$ be a semisimple profinite group, where each $S_i$
  is non-abelian finite simple.  Then $G$ is strongly complete if and
  only if there are only finitely many groups $S_i$ of each
  isomorphism type.
\end{theorem}

In particular, whenever a semisimple profinite group $G$ as in the
theorem is finitely generated, there are only finitely many factors of
each isomorphism type and the group is strongly complete.  We do not
recall the full proof of Theorem~\ref{thm:SaWiMaZe}, but we explain
how the implication `$\Leftarrow$' can be derived from the following
key ingredient proved in \cite{CM_EZ96,JS_JW97}:
\begin{quote}
  $(\ddagger)$ for every $n \in \mathbb{N}$ there exists $k \in
  \mathbb{N}$ such that for every non-abelian finite simple group $S$
  whose exponent does not divide $n$,
  \[
  S = \{ x_1^{\, n} \cdots x_k^{\, n} \mid x_1, \dots, x_k \in S \}.
  \]
\end{quote}
Supposed that $G = \prod_{i \in \mathbb{N}} S_i$ as in
Theorem~\ref{thm:SaWiMaZe}, with only finitely many groups $S_i$ of
each isomorphism type, and let $H \leq_\mathrm{f} G$.  Replacing $H$
by its core we may assume that $H \trianglelefteq_\mathrm{f} G$.  Put
$n = \lvert G : H \rvert$ and choose $k \in \mathbb{N}$ as in
$(\ddagger)$.  Then $H$ contains $X = \{ x_1^{\, n} \cdots x_k^{\, n}
\mid x_1, \dots, x_k \in G \} \supseteq \prod_{i \geq j} S_i$, where
$j \in \mathbb{N}$ is such that the exponent of $S_i$ does not divide
$n$ for $i \geq j$, and hence $H$ is open in~$G$.  The existence of
the index $j$ is guaranteed by the classification of finite simple groups.

Observe that implicitly we have established the following corollary.

\begin{corollary} \label{cor:SaWiMaZe} Let $G$ be a semisimple
  profinite group and $q \in \mathbb{N}$.  Then $G^q = \langle g^q
  \mid g \in G \rangle$, the abstract subgroup generated by all $q$th
  powers, is closed in~$G$.
\end{corollary}
 
\bigskip

\noindent \textbf{\S {\thesection}.4.} The groups $[G,G]$ and $G^q$
featuring in the discussion above are examples of verbal subgroups
of~$G$.  We briefly summarise the approach of Nikolov and Segal that
led to the original proof of Theorem~\ref{thm:Serre}
in~\cite{NN_DS07a}; the theorem is derived from a `uniformity result'
concerning verbal subgroups of finite groups.

Let $d \in \mathbb{N}$, and let $w = w(X_1,\dots,X_r)$ be a group
word, i.e., an element of the free group on $r$ generators $X_1,
\dots, X_r$.  A \emph{$w$-value} in a group $G$ is an element of the
form $w(x_1, x_2, \dots, x_r)$ or $w(x_1,x_2,\dots, x_r)^{-1}$ with
$x_1,x_2,\dots, x_r \in G$.  The \emph{verbal subgroup} $w(G)$ is the
subgroup generated (algebraically, whether or not $G$ is a topological
group) by all $w$-values in $G$.  The word $w$ is \emph{$d$-locally
  finite} if every $d$-generator group $H$ satisfying $w(H) = 1$ is
finite.  Finally, a \emph{simple commutator} of length $n \geq 2$ is a
word of the form $[X_1, \dots, X_n]$, where $[X_1,X_2] = X_1^{-1}
X_2^{-1} X_1 X_2$ and $[X_1, \dots, X_n] = [[X_1, \dots, X_{n-1}],
X_n]$ for $n>2$.  It is well-known that the verbal subgroup
corresponding to the simple commutator of length $n$ is the $n$th term
of the lower central series.

Suppose that $w$ is $d$-locally finite or that $w$ is a simple
commutator.  In~\cite{NN_DS07a}, Nikolov and Segal prove that there
exists $f = f(w,d) \in \mathbb{N}$ such that in every $d$-generator
finite group $G$ every element of the verbal subgroup $w(G)$ is equal
to a product of $f$ $w$-values in $G$.  This `quantitative' statement
about (families of) finite groups translates into the following
`qualitative' statement about profinite groups.  If $w$ is $d$-locally
finite, then in every $d$-generator profinite group $G$, the verbal
subgroup $w(G)$ is open in~$G$.  From this one easily deduces
Theorem~\ref{thm:Serre}.  Similarly, considering simple commutators
Nikolov and Segal prove that each term of the lower central series of
a finitely generated profinite group $G$ is closed.

By a variation of the same method and by appealing to Zelmanov's
celebrated solution to the restricted Burnside problem, one
establishes the following result.

\begin{theorem}[Nikolov, Segal~\cite{NN_DS11}] \label{thm:q-th-power}
  Let $G$ be a finitely generated profinite group.  Then the subgroup
  $G^q = \langle g^q \mid g \in G \rangle$, the abstract subgroup
  generated by $q$th powers, is open in~$G$ for every $q \in
  \mathbb{N}$.
\end{theorem}

The approach in~\cite{NN_DS07a} is based on quite technical results
concerning products of `twisted commutators' in finite quasisimple
groups that are established in~\cite{NN_DS07b}.  The proofs
in~\cite{NN_DS11} make use of the full machinery in~\cite{NN_DS07a}.
Ultimately these results all rely on the classification of finite
simple groups.

\begin{remark}
  An independent justification of Theorem~\ref{thm:q-th-power} would
  immediately yield a new proof of Theorem~\ref{thm:Serre}.
  Furthermore, Andrei Jaikin-Zapirain has shown that the use of the
  positive solution of the restricted Burnside problem in proving
  Theorem~\ref{thm:q-th-power} is to some extent inevitable;
  see~\cite[Section~5.1]{AJ08}.
\end{remark}

We refer to the survey article~\cite{JW11} for a more thorough
discussion of the background to Serre's problem and further
information on finite-index subgroups and verbal subgroups in
profinite groups.  A comprehensive account of verbal width in groups
is given in \cite{DS09}.

\section{Nikolov and Segal's results on finite and profinite
  groups} \label{sec:NS-finite-profinite}

\noindent
\textbf{\S {\thesection}.1.} In this section we discuss the new
approach in~\cite{NN_DS12}.  The key theorem concerns normal subgroups
in finite groups.  A finite group $H$ is almost-simple if
$S \trianglelefteq H \leq \mathrm{Aut}(S)$ for some non-abelian finite
simple group~$S$.  For a finite group $\Gamma$, let $d(\Gamma)$ denote
the minimal number of generators of $\Gamma$, write
$\Gamma' = [\Gamma,\Gamma]$ for the commutator subgroup of $\Gamma$
and set
\begin{align*}
  \Gamma_0 & = \bigcap \{ T \trianglelefteq \Gamma \mid \Gamma/T
  \text{ almost-simple} \} \\ 
  & = \bigcap \{ \mathrm{C}_\Gamma(M) \mid M \text{ a non-abelian simple
    chief factor} \},
\end{align*}
where the intersection over an empty set is naturally interpreted
as~$\Gamma$.  To see that the two descriptions of $\Gamma_0$ agree,
recall that the chief factors of $\Gamma$ arise as the minimal normal
subgroups of arbitrary quotients of~$\Gamma$.  Further, we remark that
$\Gamma/\Gamma_0$ is semisimple-by-(soluble of derived length at
most~$3$), because the outer automorphism group of any simple group is
soluble of derived length at most~$3$.  This strong form of the
Schreier conjecture is a consequence of the classification of finite
simple groups.

For $X \subseteq \Gamma$ and $f \in \mathbb{N}$ we write $X^{*f} = \{
x_1 \cdots x_f \mid x_1,\dots,x_f \in X \}$.
  
\begin{theorem}[Nikolov, Segal~\cite{NN_DS12}] \label{thm:main-finite}
  Let $\Gamma$ be a finite group and $\{y_1,\dots,y_r\} \subseteq
  \Gamma$ a symmetric subset, i.e., a subset that is closed under
  taking inverses.  Let $H \trianglelefteq \Gamma$.

  \textup{(1)} If $H \subseteq \Gamma_0$ and $H \langle y_1,\dots,y_r
  \rangle = \Gamma' \langle y_1,\dots,y_r \rangle = \Gamma$ then
  \[
  \langle [h,g] \mid h \in H, g \in \Gamma \rangle = \left\{ [h_1,y_1]
    \cdots [h_r,y_r] \mid h_1,\dots,h_r \in H \right\}^{*f},
  \]
  where $f = f(r,d(\Gamma)) = O(r^6 d(\Gamma)^6)$.

  \textup{(2)} If $\Gamma = \langle y_1,\dots,y_r \rangle$ then the
  conclusion in \textup{(1)} holds without assuming $H \subseteq
  \Gamma_0$ and with better bounds on~$f$.
\end{theorem}

While the proof of Theorem~\ref{thm:main-finite} is rather involved,
the basic underlying idea is simple to sketch.  Suppose that
$\Gamma = \langle g_1, \dots, g_r \rangle$ is a finite group and $M$ a
non-central chief factor.  Then the set
$[M, g_i] = \{ [m,g_i] \mid m \in M \}$ must be `relatively large' for
at least one generator~$g_i$.  Hence $\prod_{i=1}^r [M,g_i]$ is
`relatively large'.  In order to transform this observation into a
rigorous proof one employs a combinatorial principle, discovered by
Timothy Gowers in the context of product-free sets of quasirandom
groups and adapted by Nikolay Nikolov and L\'aszl\'o Pyber to obtain
product decompositions in finite simple groups;
cf.~\cite{TG08,NN_LP11}.  Informally speaking, to show that a finite
group is equal to a product of some of its subsets, it suffices to
know that the cardinalities of these subsets are `sufficiently large'.
A precise statement of the result used in the proof of
Theorem~\ref{thm:main-finite} is the following.

\begin{theorem}[{\cite[Corollary~2.6]{LB_NN_LP08}}] 
  Let $\Gamma$ be a finite group, and let $\ell(\Gamma)$ denote the
  minimal dimension of a non-trivial $\mathbb{R}$-linear
  representation of $\Gamma$.

  If $X_1, \dots, X_t \subseteq \Gamma$, for $t \geq 3$, satisfy
  \[
  \prod_{i=1}^t \lvert X_i \rvert \geq \lvert \Gamma \rvert^t \,
  \ell(\Gamma)^{2-t},
  \]
  then $X_1 \cdots X_t = \{ x_1 \cdots x_t \mid x_i \in X_i \text{ for
  } 1 \leq i \leq t\} = \Gamma$.
\end{theorem}

A short, but informative summary of the proof of
Theorem~\ref{thm:main-finite}, based on product decompositions, can be
found in~\cite[\S10]{NN12}.

\bigskip

\noindent
\textbf{\S {\thesection}.2.} By standard compactness arguments,
Theorem~\ref{thm:main-finite} yields a corresponding result for normal
subgroups of finitely generated profinite groups.  For a profinite
group~$G$, let $d(G)$ denote the minimal number of topological
generators of $G$, write $G' = [G,G]$ for the abstract commutator
subgroup of $G$ and set
\begin{equation*}
  G_0 = \bigcap \{ T \trianglelefteq_\mathrm{o} G \mid  G/T \text{
    almost-simple} \}, 
\end{equation*}
where the intersection over an empty set is naturally interpreted
as~$G$.  As in the finite case, $G/G_0$ is semisimple-by-(soluble
of derived length at most~$3$).  For $X \subseteq G$ and
$f \in \mathbb{N}$ we write
$X^{*f} = \{ x_1 \cdots x_f \mid x_1,\dots,x_f \in X \}$ as before,
and the topological closure of $X$ in $G$ is denoted by
$\overline{X}$.

\begin{theorem}[Nikolov,
  Segal~\cite{NN_DS12}] \label{thm:main-profinite} Let $G$ be a
  profinite group and $\{y_1,\dots,y_r\} \subseteq G$ a symmetric
  subset.  Let $H \trianglelefteq_\mathrm{c} G$ be a closed normal
  subgroup.

  \textup{(1)} If $H \subseteq G_0$ and $H \overline{\langle
    y_1,\dots,y_r \rangle} = \overline{G' \langle y_1,\dots,y_r
    \rangle} = G$ then
  \[
  \langle [h,g] \mid h \in H, g \in G \rangle = \left\{ [h_1,y_1]
    \cdots [h_r,y_r] \mid h_1,\dots,h_r \in H \right\}^{*f},
  \]
  where $f = f(r,d(G)) = O(r^6 d(G)^6)$.

  \textup{(2)} If $y_1,\dots,y_r$ topologically generate $G$ then the
  conclusion in \textup{(1)} holds without assuming $H \subseteq G_0$
  and better bounds on~$f$.
\end{theorem}

\begin{proof}
  We indicate how to prove~(1).  The inclusion `$\supseteq$' is clear.
  The inclusion `$\subseteq$' holds modulo every open normal subgroup
  $N \trianglelefteq_\mathrm{o} G$, by Theorem~\ref{thm:main-finite}.
  Consider the set on the right-hand side, call it~$Y$.  Being the
  image of the compact space $H \times \dots \times H$, with $rf$
  factors, under a continuous map, the set $Y$ is closed in the
  Hausdorff space~$G$.  Hence
  \[
  \langle [h,g] \mid h \in H, g \in G \rangle \subseteq
  \bigcap\nolimits_{N \trianglelefteq_\mathrm{o} G} Y N = Y. \qedhere
  \]
\end{proof}

In particular, the theorem shows that, if $G$ is a finitely generated
profinite group and $H \trianglelefteq_\mathrm{c} G$ a closed normal
subgroup, then the abstract subgroup 
\[ 
[H,G] = \langle [h,g] \mid h \in H, g \in G \rangle
\]
is closed.  Thus $G'$ and more generally all terms $\gamma_i(G)$ of
the abstract lower central series of $G$ are closed; these
consequences were already established in~\cite{NN_DS07a}.
Theorem~\ref{thm:further-results-3}, stated in the introduction,
generalises these results to more general compact Hausdorff groups.

Furthermore, one obtains from Theorem~\ref{thm:main-profinite} the
following tool for studying abstract normal subgroups of a finitely
generated profinite group~$G$, reducing certain problems more or less
to the abelian profinite group~$G/G'$ or the profinite group~$G/G_0$
which is semisimple-by-(soluble of derived length at most~$3$).

\begin{corollary}[Nikolov, Segal~\cite{NN_DS12}] \label{cor:N=G} Let
  $G$ be a finitely generated profinite group and $N \trianglelefteq
  G$ a normal abstract subgroup.  If $N G' = N G_0 = G$ then $N = G$.
\end{corollary}

\begin{proof}
  Suppose that $N G' = N G_0 = G$, and let $d = d(G)$ be the minimal
  number of topological generators of~$G$.  Then there exist $y_1,
  \dots, y_{2d} \in N$ such that
  \[
  G_0 \overline{\langle y_1, \dots, y_{2d} \rangle} = \overline{G'
    \langle y_1, \dots, y_{2d} \rangle} = G.
  \]
  Applying Theorem~\ref{thm:main-profinite}, with $H = G_0$, we obtain
  \[
  [G_0,G] \subseteq \langle [G_0,y_i] \cup [G_0,y_i^{\, -1}] \mid 1
  \leq i \leq 2d \rangle \subseteq N,
  \]
  hence $G = NG' = N [NG_0,G] = N$.
\end{proof}

Using Corollaries~\ref{cor:N=G} and~\ref{cor:SaWiMaZe}, it is not
difficult to derive Theorem~\ref{thm:Serre}.

\begin{proof}[Proof of Theorem~\ref{thm:Serre}]
  Let $H \leq_\mathrm{f} G$ be a finite-index subgroup of the finitely
  generated profinite group~$G$.  Then its core $N = \bigcap_{g \in G}
  H^g \trianglelefteq_\mathrm{f} G$ is contained in~$H$, and it
  suffices to prove that $N$ is open in~$G$.  The topological closure
  $\overline{N}$ is open in~$G$; in particular, $\overline{N}$ is a
  finitely generated profinite group and without loss of generality we
  may assume that $\overline{N} = G$.

  Assume for a contradiction that $N \lneqq G$.  Using
  Corollary~\ref{cor:N=G}, we deduce that
  \begin{equation} \label{equ:alternative}
    N G' \lneqq G \qquad \text{or} \qquad N G_0 \lneqq G.
  \end{equation}

  Setting $q = \lvert G : N \rvert$, we know that the abstract
  subgroup $G^q = \langle g^q \mid g \in G \rangle$ generated by all
  $q$th powers is contained in~$N$.  By
  Theorem~\ref{thm:main-profinite}, the abstract commutator subgroup
  $G' = [G,G]$ is closed, thus the subgroup $G'G^q$ is closed in~$G$.
  Hence $G/ G'G^q$, being a finitely generated abelian profinite group
  of finite exponent, is finite and discrete.  As $NG'/G'G^q$ is dense
  in $G/G'G^q$, we deduce that $NG' = NG'G^q = G$.

  It suffices to prove that $NG_0 = G$ in order to obtain a
  contradiction to~\eqref{equ:alternative}.  Factoring out by $G_0$,
  we may assume without loss of generality that~$G_0 = 1$.  Then $G$
  has a semisimple subgroup $T \trianglelefteq_\mathrm{c} G$ such that
  $G/T$ is soluble.  From $G = NG'$ we see that $G/NT$ is perfect and
  soluble, and we deduce that
  \begin{equation} \label{equ:NT=G} 
     NT = G.
  \end{equation}
  Corollary~\ref{cor:SaWiMaZe} shows that $T^q \leq_\mathrm{c} T$.
  Factoring out by $T^q$ (which is contained in~$N$) we may assume
  without loss of generality that $T^q = 1$.  The definition of $G_0$
  shows that $T$ is a product of non-abelian finite simple groups,
  each normal in $G$, of exponent dividing~$q$.  Using the
  classification of finite simple groups, one sees that the finite
  simple factors of $T$ have uniformly bounded order.  Thus
  $G/\mathrm{C}_G(T)$ is finite.  As $T \cap \mathrm{C}_G(T) = 1$, we
  conclude that $T$ is finite, thus $T \cap N \leq_\mathrm{c} G$.
  This shows that
  \[
  T = [T,G] = [T,\overline{N}] \leq \overline{[T,N]} \leq T \cap N,
  \]
  and \eqref{equ:NT=G} gives $N = NT = G$.
\end{proof}

\bigskip

\noindent
\textbf{\S {\thesection}.3.} The methods developed in~\cite{NN_DS12}
lead to new consequences for abstract quotients of finitely generated
profinite groups and, more generally, compact Hausdorff topological
groups.  In the introduction we stated three such results:
Theorems~\ref{thm:further-results-1}, \ref{thm:further-results-2} and
\ref{thm:further-results-3}.

We finish by indicating how Corollary~\ref{cor:N=G} can be used to see
that, for profinite groups, the assertion in
Theorem~\ref{thm:further-results-1} reduces to the following special
case.

\begin{theorem} \label{thm:semisimple-case}
  A finitely generated semisimple profinite group does not have
  countably infinite abstract images.
\end{theorem}

Nikolov and Segal prove Theorem~\ref{thm:semisimple-case} via a
complete description of the maximal normal abstract subgroups of a
strongly complete semisimple profinite group; see
\cite[Theorem~5.12]{NN_DS12}.  Their argument relies, among other
things, on work of Martin Liebeck and Aner Shalev~\cite{ML_AS01} on
the diameters of non-abelian finite simple groups.

\begin{proof}[Proof of `Theorem~\ref{thm:semisimple-case} implies
  Theorem~\ref{thm:further-results-1} for profinite groups']
  For a contradiction, assume that $\Gamma = G/N$ is an infinite
  finitely generated abstract image of a profinite group~$G$.
  Replacing $G$ by the closed subgroup generated by any finite set
  mapping onto a generating set of $\Gamma$, we may assume that $G$ is
  topologically finitely generated, hence strongly complete by
  Theorem~\ref{thm:Serre}.  Let
  $\Gamma_1 = \bigcap_{\Delta \trianglelefteq_\mathrm{f} \Gamma}
  \Delta$.
  Then $\Gamma / \Gamma_1$ is a countable residually finite image
  of~$G$, and Corollary~\ref{cor:no-inf-res-fin} shows that
  $\lvert \Gamma : \Gamma_1 \rvert < \infty$.  Replacing $\Gamma$ by
  $\Gamma_1$ and $G$ by the preimage of $\Gamma_1$ in~$G$, we may
  assume that $\Gamma$ has no non-trivial finite images.

  We apply Corollary~\ref{cor:N=G} to $G$ and the normal abstract
  subgroup $N \trianglelefteq G$ with $G/N = \Gamma$.  Since $\Gamma$
  is finitely generated and has no non-trivial finite images, we
  conclude that $\Gamma' = \Gamma$ and hence $N G' = G$.  Thus $N G_0
  \lneqq G$.  Consequently, $G/G_0$ maps onto the infinite finitely
  generated perfect quotient $\Gamma_0 = G/NG_0$ of~$\Gamma$.  We may
  assume that $G_0 =1$ so that $G$ is semisimple-by-soluble.  Since
  soluble groups do not map onto perfect groups, we may assume that
  $G$ is semisimple.
\end{proof}

\section{Abstract versus continuous cohomology} \label{sec:coho}

In this section we interpret Theorem~\ref{thm:Serre} in terms of
cohomology groups and indicate some work in preparation; for more
details we refer to a forthcoming joint paper~\cite{YB_AJ_BK} with
Yiftach Barnea and Jaikin-Zapirain.

\bigskip

\noindent
\textbf{\S\thesection.1.}  Let $\Gamma$ be a dense abstract subgroup
of a profinite group~$G$, and let $M$ be a continuous finite
$G$-module.  For $i \in \mathbb{N}_0$, the $i$-dimensional continuous
cohomology group $\mathrm{H}^i_\mathrm{cont}(G,M)$ can be defined as
the quotient $\mathrm{Z}^i_\mathrm{cont}(G,M) /
\mathrm{B}^i_\mathrm{cont}(G,M)$ of continuous $i$-cocycles modulo
continuous $i$-coboundaries with values in~$M$.  Alternatively, the
continuous cohomology can be described via the direct limit
\[
\mathrm{H}_\mathrm{cont}^i(G,M) = \varinjlim_{U
  \trianglelefteq_\mathrm{o} G} \mathrm{H}^i(G/U,M^U),
\]
where $U$ runs over all open normal subgroups, $M^U$ denotes the
submodule of invariants under~$U$, and $\mathrm{H}^i(G/U,M^U)$ denotes
the ordinary cohomology of the finite group~$G/U$.  Restriction
provides natural maps
\begin{equation} \label{equ:restr} \mathrm{H}^i_\mathrm{cont}(G,M)
  \rightarrow \mathrm{H}^i(\Gamma,M), \quad i \in \mathbb{N}_0.
\end{equation}
Here  $\mathrm{H}^i(\Gamma,M)$ is the ordinary cohomology, or
equivalently the continuous cohomology of $\Gamma$ equipped with the
discrete topology.

In fact, we concentrate on the case $\Gamma = G$, regarded as an
abstract group, and we write $\mathrm{H}^i_\mathrm{disc}(G,M) =
\mathrm{H}^i(\Gamma,M)$ for the cohomology of the abstract group~$G$.
Nikolov and Segal's solution to Serre's problem can be reformulated as
follows.

\begin{theorem}[Nikolov, Segal] \label{thm:Nik-Seg-coho}
  Let $G$ be a finitely generated profinite group and $M$ a continuous
  finite $G$-module.  Then $\mathrm{H}^1_\mathrm{cont}(G,M)
  \rightarrow \mathrm{H}^1_\mathrm{disc}(G,M)$ is a bijection, and
  $\mathrm{H}^2_\mathrm{cont}(G,M) \rightarrow
  \mathrm{H}^2_\mathrm{disc}(G,M)$ is injective.
\end{theorem}

It is natural to ask for analogous results in higher dimensions, but
the situation seems to become rather complicated already in dimensions
$2$ and~$3$.  Currently, very little is known, even at a conjectural
level.  Some basic results, regarding pro-$p$ groups, can be found in
\cite{FKRS08}; see also the references therein.

\bigskip

\noindent
\textbf{\S\thesection.2.}  In~\cite[I.2.6]{JPS97}, Serre introduced a
series of equivalent conditions to investigate the
maps~\eqref{equ:restr}, in a slightly more general setting.  For our purpose
the following formulation is convenient: for $n \in \mathbb{N}_0$ and
$p$ a prime, we define the property
\begin{description}
\item[$\mathsf{E}_n(p)$] for every continuous finite $G$-module $M$ of
  $p$-power cardinality, the natural map
  $\mathrm{H}_\mathrm{cont}^i(G,M) \to
  \mathrm{H}^i_\mathrm{disc}(G,M)$ is bijective for $0 \leq i \leq n$.
\end{description}
We say that $G$ satisfies $\mathsf{E}_n$ if $\mathsf{E}_n(p)$ holds
for all primes~$p$.  Since cohomology `commutes' with taking direct
sums in the coefficients, $G$ satisfies $\mathsf{E}_n$ if and only if,
for every continuous finite $G$-module~$M$, the maps~\eqref{equ:restr}
are bijections for $0 \leq i \leq n$.  We say that $G$ is
\emph{cohomologically $p$-good} if $G$ satisfies $\mathsf{E}_n(p)$ for
all~$n \in \mathbb{N}_0$.  The group $G$ is called
\emph{cohomologically good} if it is cohomologically $p$-good for all
primes~$p$; cf.~\cite[I.2.6]{JPS97}.

The following proposition from~\cite{YB_AJ_BK} provides a useful tool
for showing that a group has property~$\mathsf{E}_n(p)$.

\begin{proposition} \label{pro:est-closed} A profinite group $H$
  satisfies $\mathsf{E}_n(p)$ if there is an abstract short exact
  sequence
  \[
  1 \longrightarrow N \longrightarrow H \longrightarrow G
  \longrightarrow 1,
  \]
  where $N$ and $G$ are finitely generated profinite groups
  satisfying $\mathsf{E}_n(p)$.
\end{proposition}

We discuss a natural application.  Recall that the \emph{rank} of a
profinite group $G$ is
\[
\mathrm{rk}(G) = \sup \{ d(H) \mid H \leq_\mathrm{c} G \},
\]
where $d(H)$ is the minimal number of topological generators of~$H$.
Using the classification of finite simple groups, it can be shown that
if $G$ is a profinite group all of whose Sylow subgroups have rank at
most $r$ then $\mathrm{rk}(G) \leq r+1$.  Furthermore, a pro-$p$ group
$G$ has finite rank if and only if it is $p$-adic analytic; in this
case the dimension of $G$ as a $p$-adic manifold is bounded by the
rank: $\dim(G) \leq \mathrm{rk}(G)$.

The following result from~\cite{YB_AJ_BK}
generalises~\cite[Theorem~2.10]{FKRS08}.

\begin{theorem} \label{thm:sol-good} Soluble profinite groups of
  finite rank are cohomologically good.
\end{theorem}

It is well-known that the second cohomology groups, and hence
properties $\mathrm{E}_2(p)$ and $\mathrm{E}_2$, are intimately linked
to the theory of group extensions.  Theorem~\ref{thm:sol-good} can be
used to prove that every abstract extension of soluble profinite
groups of finite rank carries again, in a unique way, the structure of
a soluble profinite group of finite rank.  We refer
to~\cite[Chapter~8]{JW98} for a general discussion of profinite groups
of finite rank.  An example of a soluble profinite group of finite
rank that is not virtually pronilpotent can be obtained as follows:
consider
$G = \overline{\langle x \rangle} \ltimes A \cong \widehat{\mathbb{Z}}
\ltimes \widehat{\mathbb{Z}}$,
where conjugation by $x$ on a Sylow pro-$p$ subgroup
$A_p \cong \mathbb{Z}_p$ of $A = \prod_p A_p$ is achieved via
multiplication by a primitive $(p-1)$th root or unity.

\bigskip

\noindent
\textbf{\S\thesection.3} Fix a prime $p$ and consider the questions
raised above for a finitely generated pro-$p$ group~$G$.  By means of
a reduction argument, one is led to focus on cohomology groups with
coefficients in the trivial module~$\mathbb{F}_p$, i.e., the groups
\[
\mathrm{H}^i_\mathrm{cont}(G) =
\mathrm{H}^i_\mathrm{cont}(G,\mathbb{F}_p) \quad \text{and} \quad
\mathrm{H}^i_\mathrm{disc}(G) =
\mathrm{H}^i_\mathrm{disc}(G,\mathbb{F}_p) \quad \text{for $i \in
  \mathbb{N}_0$}.
\]
Indeed, $G$ satisfies $\mathsf{E}_n(p)$ if and only if, for $1 \leq i
\leq n$, every cohomology $\alpha \in \mathrm{H}_\mathrm{disc}^i(U)$,
where $U \leq_\mathrm{o} G$, vanishes when restricted to a suitable
smaller open subgroup.  It is natural to start with the cohomology
groups in dimension~$2$, and by virtue of
Theorem~\ref{thm:Nik-Seg-coho}, we may regard
$\mathrm{H}^2_\mathrm{cont}(G)$ as a subgroup of
$\mathrm{H}^2_\mathrm{disc}(G)$.  It is well-known that these groups
are intimately linked with continuous, respectively abstract, central
extensions of a cyclic group $C_p$ by~$G$.

First suppose that the finitely generated pro-$p$ group $G$ is not
finitely presented as a pro-$p$ group.  For instance, the profinite
wreath product $C_p \,\hat{\wr}\, \mathbb{Z}_p = \varprojlim_i C_p \wr
C_{p^i}$ has this property.  Then it can be shown that there exists a
central extension $1 \rightarrow C_p \rightarrow H \rightarrow G
\rightarrow 1$ such that $H$ is not residually finite, hence cannot be
a profinite group; compare~\cite[Theorem~A]{FKRS08}.  It follows that
$\mathrm{H}^2_\mathrm{cont}(G)$ is strictly smaller than
$\mathrm{H}^2_\mathrm{disc}(G)$, and $G$ does not
satisfy~$\mathsf{E}_2(p)$.  This raises a natural question.
\begin{problem*}
  Characterise finitely presented pro-$p$ groups that satisfy
  $\mathsf{E}_2(p)$.
\end{problem*}

Groups of particular interest, for which the situation is not yet
fully understood, include on the one hand non-abelian free pro-$p$
groups and on the other hand $p$-adic analytic pro-$p$ groups.
Indeed, a rather intriguing theorem of Aldridge
K.~Bousfield~\cite[Theorem~11.1]{AB92} in algebraic topology shows
that, for a non-abelian finitely generated free pro-$p$ group~$F$, the
direct sum $\mathrm{H}^2_\mathrm{disc}(F) \oplus
\mathrm{H}^3_\mathrm{disc}(F)$ is infinite.  This stands in sharp
contrast to the well-known fact that $\mathrm{H}^2_\mathrm{cont}(F) =
\mathrm{H}^3_\mathrm{cont}(F) = 0$.

In~\cite{YB_AJ_BK} we establish the following result.

\begin{theorem} Let $F$ be a finitely presented pro-$p$ group
  satisfying~$\mathsf{E}_2(p)$.  Then every finitely presented
  continuous quotient of $F$ also satisfies~$\mathsf{E}_2(p)$.
\end{theorem}

In particular, this shows that, if every finitely generated free
pro-$p$ group $F$ has $\mathrm{H}^2_\mathrm{disc}(F) =
\mathrm{H}^2_\mathrm{cont}(F) = 0$, then all finitely presented pro-$p$
groups satisfy~$\mathsf{E}_2(p)$.

We close by stating a theorem from~\cite{YB_AJ_BK} regarding the
cohomology of compact $p$-adic analytic groups that improves results
of Balasubramanian Sury~\cite{BS93} on central extensions of Chevalley
groups over non-archimedean local fields of characteristic~$0$.

\begin{theorem}
  Let $\Phi$ be an irreducible root system, and let $K$ be a
  non-archimedean local field of characteristic~$0$ and residue
  characteristic~$p$.  

  Then every open compact subgroup $H$ of the Chevalley group
  $\mathrm{Ch}_\Phi(K)$ satisfies~$E_2(p)$.  Consequently, every
  central extension of $C_p$ by $H$ is residually finite and carries,
  in a unique way, the structure of a profinite group.
\end{theorem}

\smallskip

\begin{acknowledgements*}
  I am grateful to the referee for valuable suggestions that led, for instance, to the
  inclusion of Corollary~\ref{cor:N=G-pro-p}.
\end{acknowledgements*}



\begin{thebibliography}{99}

\bibitem{LB_NN_LP08} L.~Babai, N.~Nikolov, L.~Pyber, \textit{Product
    growth and mixing in finite groups}, Proceedings of the Nineteenth
  Annual ACM-SIAM Symposium on Discrete Algorithms, 248--257, ACM, New
  York, 2008.

\bibitem{YB_AJ_BK} Y.~Barnea, A.~Jaikin-Zapirain, and B.~Klopsch,
  \textit{Abstract versus topological extensions of profinite groups},
  preprint, 2015.

\bibitem{AB92} A.~K.~Bousfield, \textit{On the $p$-adic completions of
    nonnilpotent spaces}, Trans.\ Amer.\ Math.\ Soc.\ \textbf{331}
  (1992), 335--359.

\bibitem{FKRS08} G.~A.~Fern\'andez-Alcober, I.~V.~Kazachkov,
  V.~N.~Remeslennikov, and P.~Symonds, \textit{Comparison of the
    discrete and continuous cohomology groups of a pro-$p$ group},
  Algebra i Analiz \textbf{19} (2007), 126--142; translation in: St.\
  Petersburg Math.\ J.\ \textbf{19} (2008), 961--973.

\bibitem{TG08} W.~T.~Gowers, \textit{Quasirandom groups},
  Combin.\ Probab.\ Comput.\ \textbf{17} (2008), 363–-387.

\bibitem{BH79} B.~Hartley, \textit{Subgroups of finite index in
    profinite groups}, Math.\ Z.\ \textbf{168} (1979),
  71--76.

\bibitem{AJ08} A.~Jaikin-Zapirain, \textit{On the verbal width of
    finitely generated pro-$p$ groups}, Rev.\ Mat.\ Iberoamericana
  \textbf{24} (2008), 617--630.

\bibitem{JK_13} J.~A.~Kiehlmann, \textit{Classifications of
    countably-based Abelian profinite groups}, J.\ Group Theory
  \textbf{16} (2013), 141--157.

\bibitem{CM_EZ96} C.~Mart\'inez and E.~Zelmanov, \textit{Products of
    powers in finite simple groups}, Israel J.\ Math.\ \textbf{96}
  (1996), 469--479.

\bibitem{KH_SM13} K.~H.~Hofmann and S.~A.~Morris, The structure of
  compact groups, de Gruyter, Berlin, 2013.

\bibitem{ML_AS01} M.~W.~Liebeck and A.~Shalev, \textit{Diameters of
    finite simple groups: sharp bounds and applications}, Ann.\ Math.\
  \textbf{154} (2001), 383--406.

\bibitem{NN12} N.~Nikolov, \textit{Algebraic properties of profinite
    groups}, preprint, \texttt{arXiv:1108.5130}, 2012.

\bibitem{NN_LP11} N.~Nikolov and L.~Pyber, \textit{Product
    decompositions of quasirandom groups and a Jordan type theorem},
  J.\ Eur.\ Math.\ Soc.\ \textbf{13} (2011), 1063--1077.

\bibitem{NN_DS07a} N.~Nikolov and D.~Segal, \textit{On finitely
    generated profinite groups I. Strong completeness and uniform
    bounds}, Ann.\ of Math.\ \textbf{165} (2007), 171--238.

\bibitem{NN_DS07b} N.~Nikolov and D.~Segal, \textit{On finitely
    generated profinite groups II. Products in quasisimple groups},
  Ann.\ of Math.\ \textbf{165} (2007), 239--273.

\bibitem{NN_DS11} N.~Nikolov and D.\ Segal, \textit{Powers in finite
    groups}, Groups Geom.\ Dyn.\ \textbf{5} (2011), 501--507.

\bibitem{NN_DS12} N.~Nikolov and D.\ Segal, \textit{Generators and
    commutators in finite groups; abstract quotients of compact
    groups}, Invent.\ Math.\ \textbf{190} (2012), 513--602.

\bibitem{NN_DS14} N.~Nikolov and D.\ Segal, \textit{On normal
    subgroups of compact groups}, J.\ Eur.\ Math.\ Soc.\ (JEMS)
  \textbf{16} (2014), 597--618.

\bibitem{JS_JW97} J.~Saxl and J.~S.~Wilson, \textit{A note on powers
    in simple groups}, Math.\ Proc.\ Camb.\ Philos.\ Soc.\
  \textbf{122} (1997), 91--94.

\bibitem{DS_00} D.~Segal, \textit{Closed subgroups of profinite
    groups}, Proc.\ London Math.\ Soc.\ \textbf{81} (2000),
  29--54.

\bibitem{DS09} D.~Segal, Words: notes on verbal width in groups,
  Cambridge University Press, Cambridge, 2009.
 
\bibitem{JPS97} J.-P.~Serre, Galois cohomology, Springer-Verlag,
  Berlin, 1997.

\bibitem{BS93} B.~Sury, \textit{Central extensions of $p$-adic groups;
    a theorem of Tate}, Comm.\ Algebra \textbf{21} (1993), 1203--1213.

\bibitem{JW98} J.~S.~Wilson, Profinite groups, Oxford University Press, New York, 1998.

\bibitem{JW11} J.~S.~Wilson, \textit{Finite index subgroups and verbal
    subgroups in profinite groups}, S\'eminaire Bourbaki, Vol.\
  2009/2010, Expos\'es 1012--1026.  Ast\'erisque No.\ \textbf{339}
  (2011), Exp.\ No.\ 1026.
\end{thebibliography}
\end{document}